\documentclass[11pt]{article}

 \usepackage[T1]{fontenc}
 \usepackage{tgtermes}
 \usepackage{amsmath}
 \usepackage{graphicx}
\usepackage{framed}
\usepackage[normalem]{ulem}
\usepackage{amsthm,amssymb,amsfonts}
\usepackage{enumerate}
\usepackage{diagbox}
\usepackage{xcolor}
\usepackage{subfigure}
\usepackage{epstopdf}
 \usepackage{float}

\newtheorem{proposition}{Proposition}
\newtheorem{theorem}[proposition]{Theorem}
\newtheorem{corollary}[proposition]{Corollary}

\newtheorem{definition}[proposition]{Definition}

\newcommand{\Eb}{\mathbb{E}}
\newcommand{\Fc}{\mathcal{F}}
\newcommand{\Ff}{\mathfrak{F}}
\newcommand{\Jc}{\mathcal{J}}
\newcommand{\Lc}{\mathcal{L}}
\newcommand{\Lp}{\Lc_p}
\newcommand{\Rb}{\mathbb{R}}
\newcommand{\Ds}{\text{\sf D}}
\newcommand{\R}{\mathbb{R}}

\title{Bias Reduction in Sample-Based Optimization
}

\author{Darinka Dentcheva$^\star$ and Yang Lin\thanks{Department of Mathematical Sciences, Stevens Institute of Technology, Hoboken, New Jersey, USA
 {darinka.dentcheva@stevens.edu},\; {ylin17@stevens.edu}
 }
}
\date{March 12, 2021}

\begin{document}

\maketitle

\textbf{Abstract }
We consider stochastic optimization problems which use observed data to estimate essential characteristics of the random quantities involved. Sample average approximation (SAA) or empirical (plug-in) estimation are very popular ways to use data in optimization. It is well known that sample average optimization suffers from downward bias. We propose to use smooth estimators rather than empirical ones in optimization problems. We establish consistency results for the optimal value and the set of optimal solutions of the new problem formulation. The performance of the proposed approach is compared to SAA theoretically and numerically. We analyze the bias of the new problems and identify sufficient conditions for ensuring less biased estimation of the optimal value of the true problem. At the same time, the error of the new estimator remains controlled. We show that those conditions are satisfied for many popular statistical problems such as regression models, classification problems, and optimization problems with Average (Conditional) Value-at-Risk. We have observed that smoothing the least-squares objective in a regression problem by a normal kernel leads to a ridge regression. Our numerical experience shows that the new estimators frequently exhibit also smaller variance and smaller mean-square error than those of SAA.

\textbf{Keywords }
  Kernel estimators, stochastic programming, sample average approximation, consistency, smoothing, regularization.

\section{Introduction}

Many practical situations require optimization under uncertainty. Uncertainty quantification, machine learning problems, and many problems in business and engineering lead to stochastic optimization problems. 
Plenitude of literature discusses optimization of the following general structure:
\begin{equation} 
\label{p:basic}
    \min_{u \in \mathcal{U}} \Eb \big[F(u,X)\big].
\end{equation}
In \eqref{p:basic}, $U$ is a nonempty closed subset of $\mathbb{R}^n,$ representing the feasible decisions. The random vector $X$ represents the uncertain data.
The set $\Lp(\varOmega,\Fc,P;\Rb^m)$ designates the space of integrable random vectors, defined on the probability space $(\varOmega,\Fc,P)$ with realizations in $\Rb^m$; we assume that $X\in\Lp(\varOmega,\Fc,P;\Rb^m).$
The objective function $F:\Rb^n\times\Rb^m\to\Rb$ is assumed to be sufficiently regular for the expectation to be well defined and finite valued for all $u\in U$. 

We denote the optimal value of problem \eqref{p:basic} by $\theta$ and the set of its optimal solutions by $S$.

Suppose a sample $X_1, X_2,\dots, X_N$ of $N$ realizations of the
random vector $X$ is available.
The most popular approach to problem \eqref{p:basic} is the sample average approximation, usually abbreviated as SAA, which suggests to solve  the empirical counterpart of \eqref{p:basic}:
\begin{equation}
\label{SAAestimator}
\theta_{SAA}^{(N)}=\min_{u\in U} \frac{1}{N}\sum_{i=1}^N F(u,X_{i}).
\end{equation}
It is well-known that the optimal value of the SAA problem, viewed as an estimator of $\theta $, suffers from a downward bias. The reason is the following observation
\begin{multline}
\label{simpleineq}
\quad\min_{u\in U} \mathbb{E} [F(u,X)]=\min_{u\in U} \mathbb{E} \Big[\frac{1}{N}\sum_{i=1}^N F(u, X_i)\Big] \geq \mathbb{E} \Big[\min_{u\in U} \frac{1}{N}\sum_{i=1}^N F(u, X_i)\Big]=\mathbb{E} [{\theta}_{SAA}].
\end{multline}
In other words, when solving the ideal problem \eqref{p:basic} one tries to hedge against all possible realizations of $X$ but given the sample, one actually hedges against a subset of the support of $X$ in (\ref{SAAestimator}) and, hence, we end up with a more optimistic value of the unknown value $\theta.$
The effect of this downwards bias has been discussed in the literature (\cite{norkin1998branch},\cite{mak1999monte}),\cite{bayraksan2006assessing,bayraksan2011sequential}, \cite{mainbook}).
It has also been shown that the bias decreases monotonically with the sample size and diminishes asymptotically. Again, the intuition behind this phenomenon is that the subset of the support of $X$ increases with sample size. Hence, the quality of the bound obtained from (\ref{SAAestimator}) is improved.   Some methods for bias reduction have also been proposed (see, e.g.\cite{mak1999monte}). However, those methods require additional sampling and solving many optimization problems, which is not always a practicable alternative.

While asymptotics of statistical estimators is investigated in numerous publications, we point to few references that have a closer relation to stochastic optimization and to the work in this paper. In statistics, properties of smoothed estimators are investigated in \cite{tucl,E-consistency,wied2012consistency,Silverman1978weak,einmahl2005uniform,gine2004weighted} and many other papers.
First studies in the context of stochastic optimization are contained in \cite{kall1974approximations,kankova1974optimum,kavnkova1978approximative}. New approach to asymptotic behavior of estimators using epi-convergence was presented in
\cite{dupacova1988asymptotic} and further developed in \cite{king:90,robi:96}. Normalized convergence is studied in \cite{norkin1992convergence}.  An extension of the SAA method to stochastic variational inequalities is discussed in \cite{robi:99}. Other studies analyzing SAA properties are presented in \cite{kleywegt2002sample,bayraksan2006assessing,lachout2005strong}. The recent study \cite{Pang2019} focuses on non-convex problems of composite nature which arise in machine learning.  
The SAA approach is naturally applied to solve two-stage and multi-stage stochastic optimization problems (cf.\cite{mainbook}). The asymptotic behavior of two-stage problems is investigated in \cite{pflug1998glivenko,DenRom,shapir:00}, while multi-stage problems are analyzed in \cite{pflug2016empirical,heitsch2006stability}. Second-order asymptotics is discussed in \cite{DenRom,shapir:00}. 
Recent composite optimization models involve  functionals of the following form in their objective:
\[
\varrho(u, X)=  \mathbb{E}\left[f_1 \left(u,\mathbb{E}[f_2(u,\mathbb{E}[\ldots f_k(u,\mathbb{E}[f_{k+1}(u,X)],X \right) \right]
    \ldots,X)],X)].
 \]
The empirical version of a problem with such a composite objective is analyzed in \cite{dentcheva2017statistical}, where central limit theorems have been established.
Another study addressing compositions of similar type is presented in \cite{ermoliev2013sample}, see also \cite{guigues2012sampling} for related work.
We refer to \cite{romisch2003stability} for a comprehensive review on asymptotic behavior of stochastic optimization problems and to  \cite[Chapter 5]{mainbook} for a detailed analysis of SAA models and related methods. 

Our numerical experimentation shows that in many cases of interest, the downward bias mentioned above, diminishes slowly and could still be significant at large sample sizes. Hence, it is of practical interest to reduce this bias.
Although bias reduction techniques are developed in the context of statistical estimation, they may not be practicable in the situations where the sample size $N$ as well as the dimension $n$ of the decision vector $u$ are large and when the minimization itself is computationally demanding.
Hence methods such as jackknife, where $N+1$ minimization problems have to be solved,  may be numerically prohibitive.
Therefore, a direct proposal of an alternative estimator to ${\theta}_{SAA}$  that exhibits a smaller bias, is of great interest.

The main goal of this paper is to propose new ways to use sampled data for solving an optimization problem of form \eqref{p:basic} based on smooth estimators such as kernels. The use of smoothing via kernels is known in stochastic optimization, see, e.g., \cite{ermoliev1995minimization,duchi2012randomized} and the references therein. We also refer to \cite{growe1992stochastic}, where kernel estimators are used in the context of two-stage stochastic problems. We point out that the smoothing in stochastic optimization is accomplished by applying a kernel to the expectation function with respect to the decision vector. Our proposal is different: we apply smoothing to the random objective function with respect to the data.
The advantage of the new estimators would be less biased estimation of the optimal value of the true problem. We identify sufficient conditions for the bias of the optimization problems using smooth estimators to be smaller than the bias exhibited by the SAA. 
We compare the performance of the alternative approaches to SAA theoretically and numerically to demonstrate the advantage of the proposal. Our numerical experience shows that the new estimators frequently exhibit also smaller variance than those of SAA. \\

Our paper is organized as follows. In section \ref{s:smooth_estimators}, we introduce and analyze the smooth estimators of kernel type used in a sample-based optimization problem. We show the consistency of the optimal value and the optimal solution of the new problem formulation in section \ref{s:consistency}. In section \ref{s:bias}, we obtain an upper bound of the bias in problems with smooth estimators and establish relations between the bias of those problems and the bias of their SAA counterpart \eqref{SAAestimator}.
We show the relevancy of our results in regression models and other optimization problems in section \ref{s:applications}.
Numerical experience is reported in section \ref{s:numerical_results}. 



\section{Smooth estimators}
\label{s:smooth_estimators}

Suppose a sample $X_1, X_2,\dots, X_N$ of $N$ realizations of the
random vector $X$ is available.
The estimator ${\theta}_{SAA}$ can be considered informally as a  kernel estimator of the type
\[
\min_{u\in U}\frac{1}{N} \sum_{i=1}^N \int_{\Rb^m} F(u, x )\delta(x-X_{i}) dx
\]
with $\delta (\cdot)$ being the delta function. Therefore,  one could try to replace, for each $i=1,2,\dots, N$  the delta function by a kernel function $\frac{1}{h^m}K(\frac{x-X_i}{h})$
for a suitably chosen bandwidth $h=h_N.$
This idea leads to the following kernel-based estimator:
\begin{equation}
\label{p:kernel}
\theta_K^{(N)} =
    \min_{u \in U} \frac{1}{N} \sum_{i=1}^N \int_{\Rb^m} F(u,x)K\Big(\frac{x-X_i}{h_N}\Big)\frac{1}{h_N^m}dx.
\end{equation}

We note that \emph{we do not need assume the existence of a density for the distribution of the random vector} $X$ in order to use a smooth estimator for the expected value $\Eb\big[F(u, X)\big]$ in the optimization problem. However, if the distribution of $X$ has a density, then we could employ known techniques for density estimation as a guide for choosing the necessary parameters such as the bandwidth in the kernel estimation. 

We start our analysis in a general setting, following \cite{tucl}. The empirical measure $P_N$ is defined by $P_N= \sum_{i=1}^N \frac{1}{N}\delta({X_i})$.  We consider smoothing $P_N$ by a convolution with a measure $\mu_N$, defined as follows:
\begin{equation}
\label{smoothE}
[P_N *\mu_N] F(u, X) = \int_{\Rb^m} F(u,X)\, d [P_N *\mu_N] = \frac{1}{N}\sum_{i=1}^N \int_{\Rb^m} F(u,X_i+z) \, d\mu_N(z).
\end{equation}
It is assumed that, for increasing sample size $N$, we choose a sequence of measures $\{\mu_N\}_{N=1}^\infty$, which are independent of the respective measure $P_N$. The sequence is called a \emph{proper approximate convolution identity} (cf.\cite{tucl}) if the following two conditions are satisfied: 
\begin{itemize}
\item $\mu_N$ converges weakly to point mass $\delta(0)$ when $N\to\infty$; 
\item for every $a>0$,
$\lim_{N\to\infty} |\mu_N| (\Rb^m \setminus [-a,a]^m) = 0$, with $|\mu_N|$ denoting the total variation of $\mu_N.$
\end{itemize}
Throughout the entire paper, we assume that all measures $\mu_N$ are normalized, i.e., $\mu_N(\Rb^m)=1.$
Special case of \eqref{smoothE} is the kernel estimator of form:
\[
[P_N *\mu_N] F(u, X) = \frac{1}{Nh_N^m}\sum_{i=1}^N \int_{\Rb^m} F(u,x) K\Big(\frac{x-X_i}{h_N}\Big)\, dx,
\]
where $K$ is a $m$-dimensional density function with respect to the Lebesgue measure and $h_N>0$ is a smoothing parameter such that $\lim_{N\to\infty} h_N=0$. We have 
\[
d\mu_N(x) = \frac{1}{h_N^m} K\Big(\frac{x}{h_N}\Big)\,dx.
\] 
The sequence defined this way is a proper approximate convolution identity. Indeed, for every bounded continuous function $g(x)$, we have
\begin{multline*}
\lim_{N\to\infty} \frac{1}{h_N^m} \int_{\Rb^m} g(x) K\Big(\frac{x}{h_N}\Big)\, dx = 
\lim_{N\to\infty} \int_{\Rb^m} g(h_Ny) K(y)\, dy\\ = \int_{\Rb^m} \lim_{N\to\infty} g(h_Ny) K(y)\, dy  = g(0),
\end{multline*}
where the last equation follows by virtue of the Lebesgue dominated convergence theorem. Additionally,
\[
\lim_{N\to\infty} \int_{\Rb^m\setminus [-a,a]^m} \frac{1}{h_N^m} K\Big(\frac{x}{h_N}\Big)\, dx = \lim_{N\to\infty} \int_{\Rb^m\setminus [-a/h_N,a/h_N]^m} K(y)\, dy =  0
.\]
Notice that the estimators $\mu_N$ may take more general form such as multivariate kernels where the bandwidth for $m>1$ is given by a matrix or kernels with variable bandwidth (\cite{tsybakov2008introduction,gine2016mathematical}). 

More general version of problem \eqref{p:kernel}, using measures $\mu_N$ from a proper approximate convolution identity, is the following formulation.
\begin{equation}
\label{p:smooth}
\theta_\mu^{(N)} =
    \min_{u \in U} \frac{1}{N} \sum_{i=1}^N \int_{\Rb^m} F(u,X_i+z)\;d\mu_N(z).
\end{equation}
We shall analyze the consistency of $\theta_\mu^{(N)}$  and the asymptotic behavior of the corresponding optimal solutions.


\subsection{Consistency of the optimal value and the optimal solutions for smooth estimators}
 \label{s:consistency}

Consider the collection of functions 
$\mathfrak{F}=\{ f_{u}:\Rb^m\to\Rb:\, f_u (x) = F(u,x),\; u \in U \}$.

We recall that $g:\Rb^m\to\Rb$ is called an envelope function for the class of functions $\mathfrak{F}$, if $|f_u(x)| \leq g(x) < \infty$ for all $x\in\Rb^m$ for any $f_u\in\mathfrak{F}$.
We shall pose an assumption about the modulus of continuity for the functions $F(u,\cdot)$ and, in order to define more precisely our assumptions, we work with the following notion.
\begin{definition}
Let a proper approximate convolution identity $\{\mu_N\}_{N=1}^\infty$be given. 
A function $g:\Rb^m\to\Rb$ admits a $\mu_N$-adapted non-decreasing modulus of continuity $w_g:\Rb_{+}\to\Rb$ if 
\begin{itemize}
	\item $\lim_{t\downarrow 0} w_g (t) = 0$ and $w_g(\cdot)$ is non-decreasing;
	\item for all $x,z\in\R^m$, it holds $| g(x+z) - g(x) | \leq w_g(\|z\|)$;
	\item $\lim_{N\to\infty} \int_{\Rb^m} w_g(\|z\|)\,d\mu_N(z) =0$.
\end{itemize}
\end{definition}
Note that if $g$ admits a modulus of continuity, then it also admits  a non-decreasing modulus. Furthermore, if $g$ satisfies the H\"older condition with constants $\alpha\in (0,2]$ and $L>0$,  then for all $x,z\in\Rb^m$, we have
\[
| g(x) - g(z) | \leq L \|x-z\|^\alpha.
\]
Additionally, let $K$ be a $m$-dimensional symmetric density function with finite second moment: $\int\limits_{\Rb^m}\|y\|^2 K(y)dy <\infty.$ Setting
$d\mu_N(x) = \frac{1}{h_N^m} K\Big(\frac{x}{h_N}\Big)\,dx$ as before, we observe that $g$ admits a $\mu_N$-adapted nondecreasing modulus of continuity, defined as  
 ${w_g(t) = L t^\alpha}$  for all $t>0$. In that case, we have 
\begin{align*}
\int_{\Rb^m} w_g(\|x\|)\,d\mu_N(x) & = \int_{\Rb^m} {w_g}(\|x\|) \frac{1}{h_N^m} K\Big(\frac{x}{h_N}\Big)\,dx \\
&= \int_{\Rb^m} {w_g}(\|h_Nz\|) K(z)\,dz = Lh_N^\alpha\int_{\Rb^m} \|z\|^\alpha K(z) dz
<\infty.
\end{align*}
Additionally, since $\lim_{N\to\infty} h_N =0$, we obtain
\[
\lim_{N\to\infty} \int_{\Rb^m} w_g(\|z\|)\,d\mu_N(z) = \lim_{N\to\infty} 
Lh_N^\alpha\int_{\Rb^m} \|z\|^\alpha K(z) dz =0.
\]
In particular, if $g$ is Lipschitz continuous, then it admits $K$-integrable non-decreasing modulus of continuity for kernels of this type. 
\begin{theorem} \label{t:K-consistency1}
Assume that a proper approximate convolution identity $\{\mu_N\}_{N=1}^\infty$ is  given and the following conditions are satisfied:
\begin{itemize}
    \item[(c1)] the functions $f_u(\cdot)$ admit a $\mu_N$-adapted non-decreasing modulus of continuity independent of $u$;
    \item[(c2)] $\mathfrak{F}$ has an $P_X$-integrable envelope function $g$;
    \item[(c3)] $F(\cdot ,x)$ is a continuous function for all $x\in \Rb^n$;
    \item[(c4)] the set $U$ is compact.
\end{itemize}
Then  
\begin{equation} 
    \sup_{u \in U}\bigg|\Eb\left[F(u,X)\right]-\frac{1}{N}\sum_{i=1}^N \int_{\Rb^m} F(u,X_i +z)\,d\mu_N(z)\bigg|\xrightarrow[N\to\infty]{a.s.} 0.
\end{equation}
\end{theorem}
\begin{proof}
We start from the following estimate:
    \begin{multline} \label{kernel-inequality}
        \sup_{u \in U} \left| \Eb\left[F(u,X)\right]-\frac{1}{N}\sum_{i=1}^N \int_{\Rb^m} F(u,X_i +z)\,d\mu_N(z) \right|\\
        \leq  \sup_{u \in U} \left| \Eb\left[F(u,X)\right]  -\frac{1}{N}\sum_{i=1}^{N}F(u,X_i) \right| + \qquad\qquad\\
        \sup_{u \in U} \left| \frac{1}{N}\sum_{i=1}^{N}F(u,X_i)-\frac{1}{N}\sum_{i=1}^N \int_{\Rb^m} F(u,X_i +z)\,d\mu_N(z) \right|.
    \end{multline}
Due to assumptions (c2)--(c4), the class $\Ff$ is a Glivenko-Cantelli class.
Therefore, the first term on the right hand side converges to zero when $N\to\infty$ (cf. \cite[7.48]{mainbook}). For the second term, 
denoted $\Delta_{N},$ we have the following inequality:
\begin{align*}
    \Delta_{N}  =  &\sup_{u \in U} \left| \frac{1}{N}\sum_{i=1}^{N}F(u,X_i)-\frac{1}{N}\sum_{i=1}^N \int_{\Rb^m} F(u,X_i +z)\,d\mu_N(z) \right| \\
        &\leq \frac{1}{N}\sum_{i=1}^{N}\sup_{u \in U}\left| F(u,X_i)-\int_{\Rb^m} F(u,X_i +z)\,d\mu_N(z) \right|\\
        & \leq  \frac{1}{N}\sum_{i=1}^{N}\sup_{u \in U} \int_{\Rb^m} \big| F(u,X_i)-F(u,X_i+z)\big|\, d\mu_N(z) \\
        & \leq \frac{1}{N}\sum_{i=1}^{N} \int_{\Rb^m} w(\|z\|)\, d\mu_N(z) .
\end{align*}
Hence, we obtain
\[
0\leq \Delta_N  \leq \int_{\Rb^m} w(\|z\|)\, d\mu_N(z).
\]
The function $w(\|\cdot\|)$ is $\mu_N$-adapted and, thus, we infer that
\[
\lim_{N\to\infty} \Delta_{N}
\leq \int_{\Rb^m} \lim_{N\to \infty} w(\|z\|) \, d\mu_N(z) = 0.
\]
We conclude that,
\[
\Delta_{N} = \sup_{u \in U}\left| \Eb\left[F(u,X)\right]-\frac{1}{N}\sum_{i=1}^N \int_{\Rb^m}F(u,X_i +z)\,d\mu_N(z)\right| \xrightarrow[N\to\infty]{a.s.} 0
,\]
which shows the claimed convergence.
\end{proof}

Consider a kernel function $K$ satisfying the following assumption.
\begin{itemize}
\item[(k1)] $K$ is a symmetric around zero density function with respect to the Lebesgue measure, i.e., 
\[
\int\limits_{\Rb} y_l K(y)dy_l=0\;\;\text{ for }\; l=1,\cdots,m. 
\]
\item[(k2)]  The second order moment is a finite positive number:
\[
m_2(K) = \int\limits_{\Rb^m}\|y\|^2 K(y)dy.
\]
\end{itemize}
Note that the quantity $\bar{m}_\alpha(K) = \int\limits_{\Rb^m}\|y\|^\alpha K(y)dy$ is finite for all $\alpha\in [0,2]$ under assumption (k2).

From now till the end of this paper, we shall omit the limits of the integrals in the presentation. Unless otherwise noted all integrals will be taken over $\Rb^m.$
\begin{corollary}
\label{c:consistency1}
Assume that a kernel function $K$ satisfying (k1)--(k2) is given, $\lim_{N\to\infty} h_N=0$, and assumptions (c2)--(c4) are satisfied. Additionally, suppose  
\begin{itemize}
    \item[(c1k)] the functions $f_u(\cdot)$ admit a $K$-integrable non-decreasing modulus of continuity $w(\cdot)$ independent of $u$.
\end{itemize}
Then 
\begin{equation} 
   \sup_{u \in U}\bigg|\Eb\left[F(u,X)\right]-\frac{1}{N}\sum_{i=1}^N \int F(u,z)K\Big(\frac{z-X_i}{h_N}\Big)\frac{1}{h_N}dz\bigg| \xrightarrow[N\to\infty]{a.s.} 0.
\end{equation}
\end{corollary}
\begin{proof}
Notice that the assumption about (c1k) is more explicit and easier to verify than (c1). It implies the existence of a nondecreasing function $w$ such that
$\lim_{t\downarrow 0} w (t) = 0$, the integral $\int w(\|z\|) K(z) dz$ is finite, and for all $x,z\in\Rb^m$, it holds $| f_u(x+z) - f_u(x) | \leq w(\|z\|)$. 

The proof of Theorem~\ref{t:K-consistency1} needs to be modified only for the upper estimate on $\Delta_{N}$. We obtain in the same way as previously 
\begin{align*}
    \Delta_{N} \leq \frac{1}{N}\sum_{i=1}^{N} \int w(h_N\|z\|) K(z)dz.
\end{align*}
Without loss of generality, we assume $h_N\leq 1$. Hence, we use the monotonicity of $w(\cdot)$ to infer the inequalities
\[
0\leq \Delta_N  \leq \int w(\|z\|) K(z)dz.
\]
The function $w(\|\cdot\|)$ is $K$-integrable and, thus, we can apply Lebesgue dominated convergence theorem and obtain that
\[
\lim_{h_N\downarrow 0} \Delta_{N}
\leq \int \lim_{h_N\downarrow 0} w(h_N\|z\|) K(z)dz = 0.
\]
The last equations follows by the definition of modulus of continuity.
All other arguments remain the same.    
\end{proof}

We denote the solutions set of problem \eqref{p:smooth} by $S_\mu^{(N)}$, while
the solution set of \eqref{p:kernel} is denoted by $S_K^{(N)}.$
For two sets, $A, B\subset\Rb^n$, we consider the deviation of $A$ from $B$, defined as follows: 
\[
\Ds(A,B) = \sup_{x\in A} d(x,B) = \sup_{x\in A} \;\inf_{y\in B} \|x-y\|.
\]
The next statement can be inferred from (\cite[Theorem 2.4]{mainbook}) using Theorem~\ref{t:K-consistency1} and Corollary~\ref{c:consistency1}. We provide more detailed proof due to the key role of the statement and for convenience of the Reader.
\begin{theorem} 
\label{t:K-consistency2} {~~}
\begin{itemize}
    \item[(i1)]
Under the assumptions of Theorem \ref{t:K-consistency1}, the estimator $\theta_\mu^{(N)}$ of $\theta$ is strongly consistent, i.e.,
$\theta_\mu^{(N)} \xrightarrow[N\to\infty]{a.s.} \theta .$
Furthermore, $\Ds(S_\mu^{(N)}, S)\xrightarrow[N\to\infty]{a.s.} 0.$
\item[(i2)]
Under the assumptions of Corollary~\ref{c:consistency1},  $\theta_K^{(N)}$ is a strongly consistent estimator of $\theta$, i.e.,
$\theta_K^{(N)} \xrightarrow[N\to\infty]{a.s.} \theta$. Additionally,  
$\Ds(S_K^{(N)}, S)\xrightarrow[N\to\infty]{a.s.} 0.$
\end{itemize}
\end{theorem}
\begin{proof}
Under the continuity assumptions for the function $F$ and the compactness of $U$, problems \eqref{p:basic}, \eqref{p:kernel}, and \eqref{p:smooth} are solvable and their solution sets are non-empty subsets of $U$.

We shall show statement (i2) for the kernel-based estimator; the counterpart statement (i1) for the smooth estimators follow in the same way. 
Let $\hat{u}\in S$ and $u_K^{(N)}\in S_K^{(N)}.$  
Using Corollary~\ref{c:consistency1}, we obtain the following a.s. estimates
\begin{align*}
\theta_K^{(N)} -\theta & \leq \frac{1}{N}\sum_{i=1}^N\int F(\hat{u},X_i+h_Nz)K(z)dz - \Eb[F(\hat{u}, X)]\\
& \leq \sup_{u\in U}\Big| \frac{1}{N}\sum_{i=1}^N\int F({u},X_i+h_Nz)K(z)dz - \Eb[F({u}, X)]\Big|.
\end{align*}
Using the same argument,
\begin{align*}
\theta -\theta_K^{(N)} & \leq \Eb[F(u_K^{(N)}, X)] - \frac{1}{N}\sum_{i=1}^N \int F(u_K^{(N)},X_i+h_Nz)K(z)dz \\
& \leq \sup_{u\in U}\Big| \frac{1}{N}\sum_{i=1}^N\int F({u},X_i+h_Nz)K(z)dz - \Eb[F({u}, X)]\Big|\\
\end{align*}
Therefore, $|\theta_K^{(N)} -\theta| \leq \sup_{u\in U}\Big| \frac{1}{N}\sum_{i=1}^N\int F({u},X_i+h_Nz)K(z)dz - \Eb[F({u}, X)]\Big|$ a.s. and the first claim follows from Corollary~\ref{c:consistency1}. 

Consider the distance $\Ds(S_K^{(N)}, S)$  and assume that it does not converge to zero as claimed. This means that a sequence $u_K^{(N_j)}\in S_K^{(N_j)}$, $j=1,2,\dots$, exists such that for some $\varepsilon >0$, it holds $P\big(d(u_K^{(N_j)}, S)\geq \varepsilon)\geq p>0$ for all $N_j$ in that sequence. The points $u_K^{(N_j)}\in U$, thus, the sequence $\{u_K^{(N_j)}\}$
has a convergent subsequence, denoted $\Jc$, with accumulation point $u^*\in U$. Due to the continuity of the distance function,  we obtain $P(d(u^*, S)\geq \varepsilon > 0)\geq p.$
On the other hand, the following relations hold: 
\begin{multline*}
\theta =\!\lim_{j\in\Jc,\; h_N\downarrow 0} \theta_K^{(N_j)} 
\geq  \lim_{j\in\Jc,\; h_N\downarrow 0} \frac{1}{N_j}\sum_{i=1}^{N_j}\int F(u^*\!, X_i+h_{N_j}z ) K(z)dz = \Eb [F(u^*, X) ],
\end{multline*}
implying $P(d(u^*\!, S)=0)=1,$ which is a contradiction. 
\end{proof}

\section{Bias relations for the kernel smoothing}
\label{s:bias}

We shall relate the bias of the estimator $\theta_K^{(N)}$ to the bias of the SAA.
We assume the latter bias to be strictly negative. 
Let us fix an optimal solution $\hat{u}\in S.$ We shall assume a particular form of the modulus of continuity, which fits to the applications discussed in the paper. First, we obtain an upper bound to the bias and the mean-square error of the estimate. 
\begin{theorem}
\label{t:kernel-bias-upperbound}
Assume that the kernel $K$ satisfies (k1)-(k2) and the function $F(\hat{u},\cdot)$ in problem \eqref{p:basic} admits a modulus of continuity of form $w(t)= \sum_{j=1}^\ell L_jt^{\alpha_j} $ with $\alpha_j\in(0,2]$, $\ell\geq 1$. 
Then constants $L>0$ and $\alpha\in(0,2]$ exist, so that the bias of the kernel estimator $\theta_K^{(N)}$ is bounded from above by 
$L h_N^\alpha$. 
\end{theorem}
\begin{proof}
The following chain of inequalities holds:
\begin{align}
\Eb[\theta_K^{(N)}] - \theta  
& =
\Eb\Big[\min_{u\in U}\frac{1}{N} \sum_{i=1}^N \int  F(u,X_i+h_Nz)K(z)dz \Big]-
\min_{u\in U} \frac{1}{N} \sum_{i=1}^N \Eb[F(u,X_i)]\notag\\
& \leq 
\Eb\Big[\frac{1}{N} \sum_{i=1}^N \int F(\hat{u},x_i+h_Nz) K(z)dz \Big] - \frac{1}{N} \sum_{i=1}^N  \Eb[F(\hat{u},X_i)]\notag\\
& = 
\frac{1}{N} \sum_{i=1}^N \Eb\Big[\int \Big( F(\hat{u},X_i+h_Nz)- F(\hat{u},X_i)\Big) K(z)dz \Big]\notag\\
& \leq 
\sum_{j=1}^\ell L_jh_N^{\alpha_j} \int \|z\|^{\alpha_j} K(z)dz \label{upperbound-intermed}. 
\end{align}
Without loss of generality, we may assume that $h_N<1$. 
Therefore, we select $\alpha=\min_{1\leq j\leq \ell} \alpha_j$ and obtain
\[
\sum_{j=1}^\ell L_jh_N^{\alpha_j} \int \|z\|^{\alpha_j} K(z)dz \leq h_N^\alpha \sum_{j=1}^\ell L_j\bar{m}_{\alpha_j}(K).
\] 
All terms $L_j\bar{m}_{\alpha_j}(K)$ are finite under the assumptions (k1)-(k2). We define $L= \sum_{j=1}^\ell L_j\bar{m}_{\alpha_j}(K)$
and using relations \eqref{upperbound-intermed}, we obtain
\begin{equation}
\label{bound-lip}
 \Eb[\theta_K^{(N)}] - \theta \leq Lh_N^\alpha   
\end{equation} 
as stated. 
\end{proof}
We notice that the constant $L$ is independent of the sample size; it depends on the structure of the function $F(\hat{u},\cdot)$. The only way the upper bound of the bias depends on the sample is by the choice of the bandwidth $h_N$. 
The bound \eqref{bound-lip} has the following implication.
\begin{corollary}
Under the conditions of Theorem~\ref{t:kernel-bias-upperbound}, a number $h^*>0$ exists, such that whenever $h_N\in (0,h^*)$ the following relation holds 
\begin{equation}
\Eb[\theta_{K}^{(N)}] -\theta  \leq |\Eb[\theta_{SAA}^{(N)}] -\theta |.
\label{bias-interval-upper}
\end{equation}
\end{corollary}
Indeed, having in mind that $\theta - \Eb[\theta_{SAA}^{(N)}]>0$, an upper bound $h^*$ exists such that
$L h_N^\alpha \leq \theta - \Eb[\theta_{SAA}^{(N)}]$
whenever $0\leq h_N\leq h^*$.
\begin{theorem}
\label{t:kernel-bias-convex} 
Let the function $F(u,\cdot)$ be convex for any fixed $u\in U$.
Then $\theta_{K}^{(N)}\geq \theta_{SAA}^{(N)}$ for any sample realization.
If additionally, the conditions of Theorem \ref{t:kernel-bias-upperbound} are satisfied, then a positive number $h^*$ exists, such that whenever $h_N\in (0,h^*)$ the following relations holds with the constants $L$ and $\alpha$ from Theorem~\ref{t:kernel-bias-upperbound}
\begin{align}
 &\; |\Eb[\theta_{K}^{(N)}] -\theta | \leq |\Eb[\theta_{SAA}^{(N)}] -\theta |.
\label{bias-interval-final}\\
 &\; \Eb|\theta_{K}^{(N)} -\Eb[\theta_{K}^{(N)}] | \leq \Eb |\theta_{SAA}^{(N)}-\Eb[\theta_{SAA}^{(N)}] | +  Lh_N^\alpha,\label{dispertion}\\
 &\; \Big(\Eb[(\theta_{K}^{(N)} -\Eb[\theta_{K}^{(N)}] )^2]\Big)^{\frac{1}{2}} \leq \Big(\Eb (\theta_{SAA}^{(N)}-\Eb[\theta_{SAA}^{(N)}] )^2\Big)^{\frac{1}{2}} +  Lh_N^\alpha\label{deviation}\\
&\;\Big(\Eb\Big[\big(\theta_K^{(N)} - \theta  \big)^2\Big]\Big)^{\frac{1}{2}}\leq \Big(\Eb\Big[\big(\theta_{SAA}^{(N)} - \theta  \big)^2\Big]\Big)^{\frac{1}{2}} + Lh_N^{\alpha}.
\label{MSE-upper}
\end{align}
\end{theorem}
\begin{proof}
Substituting $z= \frac{x-X_i}{h_N}$, we obtain for any $i=1,\dots, N$
\begin{equation}
\label{in:convex}
\int F(u,x)K\Big(\frac{x-X_i}{h_N}\Big)\frac{1}{h_N^m}dx = \int F(u,X_i+h_Nz)K(z)dz \geq F(u,X_i),
\end{equation}
where the last relation follows from Jensen's inequality and assumption (k1).
For any realization of a sample, the following inequality is true
\begin{equation}
\label{bias-bound}
\begin{aligned}
\theta_K^{(N)}-&\theta_{SAA}^{(N)}\\ &= 
\min_{u\in U} \Bigg[\frac{1}{N}\sum_{i=1}^N \int F(u,x)K\Big(\frac{x-X_i}{h_N}\Big)\frac{1}{h_N^m}dx\Bigg] -
\min_{u\in U} \Big[\frac{1}{N}\sum_{i=1}^N F(u, X_i)\Big]\\
&\geq 
\min_{u\in U} \Bigg[\frac{1}{N}\sum_{i=1}^N \int F(u,x)K\Big(\frac{x-X_i}{h_N}\Big)\frac{1}{h_N^m}dx -
\frac{1}{N}\sum_{i=1}^N F(u, X_i)\Bigg]\\
& = \min_{u\in U} \frac{1}{N}\sum_{i=1}^N \Big( \int F(u,X_i+h_Nz)K(z) dz -  F(u, X_i)\Big)\geq 0.
\end{aligned}
\end{equation}
The last inequality in the chain follows by \eqref{in:convex}. Thus 
\[
\theta_K^{N} \geq \theta_{SAA}^{(N)}\;\text{ and }\;  \theta_K^{N} -\theta \geq \theta_{SAA}^{(N)} -\theta  \quad\text{a.s.},
\]
showing the first claim of the statement. This implies
\begin{equation}
\label{kernel-bias-lowerbound}
\Eb[\theta_K^{(N)}] -\theta \geq \Eb[\theta_{SAA}^{(N)}] -\theta.
\end{equation}
Putting \eqref{kernel-bias-lowerbound}  together with \eqref{bias-interval-upper}, we obtain
\[
\Eb[\theta_{SAA}^{(N)}]-\theta  \leq \Eb[\theta_{K}^{(N)}] -\theta  \leq \theta  - \Eb[\theta_{SAA}^{(N)}], 
\]
or equivalently  $|\Eb[\theta_{K}^{(N)}] -\theta  |\leq |\Eb[\theta_{SAA}^{(N)}] -\theta |$ as claimed in \eqref{bias-interval-final}. 

To show the inequality between the dispersions, we observe first that for any solution $\tilde{u}$ of problem \eqref{SAAestimator}, the following chain of inequalities apply:
\begin{align}
\theta_K^{(N)} - \theta_{SAA}^{(N)} 
& =
\min_{u\in U}\frac{1}{N} \sum_{i=1}^N \int  F(u,X_i+h_Nz)K(z)dz -
\min_{u\in U} \frac{1}{N} \sum_{i=1}^N F(u,X_i)\notag\\
& \leq 
\frac{1}{N} \sum_{i=1}^N \int F(\tilde{u},x_i+h_Nz) K(z)dz  - \frac{1}{N} \sum_{i=1}^N  F(\tilde{u},X_i)\notag\\
& = {}
\frac{1}{N} \sum_{i=1}^N \int \Big( F(\tilde{u},X_i+h_Nz)- F(\tilde{u},X_i)\Big) K(z)dz \notag\\
& \leq 
\sum_{j=1}^\ell L_jh_N^{\alpha_j} \int \|z\|^{\alpha_j} K(z)dz \leq Lh_N^\alpha, \label{delta-exp}
\end{align}
where the constants $L$ and $\alpha$ are the same as in \eqref{bias-interval-upper}.
Therefore, we have
\begin{equation}
\label{basic_pointwise}
\theta_{SAA}^{(N)}\leq \theta_K^{(N)} \leq  \theta_{SAA}^{(N)} + Lh_N^\alpha.
\end{equation}
This entails additionally, 
\[
-\Eb[\theta_{SAA}^{(N)}]-Lh_N^\alpha \leq -\Eb[\theta_K^{(N)}] \geq  -\Eb[\theta_{SAA}^{(N)}].
\]
Adding the two relations together, we get
\begin{equation}
\label{mainpointwise}
\theta_{SAA}^{(N)}-\Eb[\theta_{SAA}^{(N)}]-Lh_N^\alpha \leq \theta_K^{(N)} -\Eb[\theta_K^{(N)}]\leq  \theta_{SAA}^{(N)}-\Eb[\theta_{SAA}^{(N)}] + Lh_N^\alpha.
\end{equation}
We consider $U=\theta_{K}^{(N)}-\Eb[\theta_{K}^{(N)}]$ and $V=\theta_{SAA}^{(N)}-\Eb[\theta_{SAA}^{(N)}]$.
We infer that 
\begin{align*}
\Eb\big|\theta_K^{(N)} -\Eb[\theta_K^{(N)}]\big| & = \Eb|U| \leq 
\Eb|V| +
\Eb|U-V| \leq \Eb|V| + Lh_N^\alpha\\
& = 
\Eb\big|\theta_{SAA}^{(N)} -\Eb[\theta_{SAA}^{(N)}]\big| +Lh_N^\alpha
\end{align*}
This proves inequality \eqref{dispertion}.\\

Now, we consider $U$ and $V$ as random variables in $\mathcal{L}_2$. Using \eqref{mainpointwise} again and the triangle inequality for the norm, we get
\begin{align*}
\Big(\Eb[(\theta_{K}^{(N)} -\Eb[\theta_{K}^{(N)}] )^2]\Big)^{\frac{1}{2}}
& =\|U\|_{\Lc_2} \leq \|V\|_{\Lc_2}+ \|U-V\|_{\Lc_2} \\
&\leq \|V\|_{\Lc_2}+ Lh_N^\alpha = \Big(\Eb[(\theta_{SAA}^{(N)} -\Eb[\theta_{SAA}^{(N)}] )^2]\Big)^{\frac{1}{2}} + Lh_N^\alpha. 
\end{align*}  
This proves \eqref{deviation}.\\

Now, we turn to the upper bound on the error of the kernel estimator. 
Using inequalities \eqref{basic_pointwise}, we get
\[
|\theta_{K}^{(N)} -\theta| \leq |\theta_{SAA}^{(N)}-\theta| + 
|\theta_{K}^{(N)} -\theta_{SAA}^{(N)}| \leq |\theta_{SAA}^{(N)}-\theta| + Lh_N^\alpha.
\]
Denoting $Y=|\theta_{K}^{(N)} -\theta|$ and $W=|\theta_{SAA}^{(N)}-\theta|$, we view $Y$ and $W$ as two non-negative random variables in $\mathcal{L}_2$.
Recall that $\mathcal{L}_2$ is a Banach lattice and therefore, the relation between $Y$ and $W$ entails the same relation between their 
$\Lc_2$-norms, i.e.,
\[
\Big(\Eb[(\theta_{K}^{(N)} -\theta])^2]\Big)^{\frac{1}{2}}
= \|Y\|_{\Lc_2}\leq \|W\|_{\Lc_2} + Lh_N^\alpha = 
\Big(\Eb[(\theta_{SSA}^{(N)} -\theta])^2]\Big)^{\frac{1}{2}} + Lh_N^\alpha.
\]
This shows \eqref{MSE-upper} and completes the proof.
\end{proof}

Some remarks are in order. 
First, we note that every convex function is locally Lipschitz-continuous. However, the existence of the continuity modulus is essential and it would correspond to the existence of a global Lipschitz constant. This property is satisfied for functions, which are piecewise linear, for example. Global Lipschitz continuity would be present also when the random vector $X$ has a bounded support. 

The comparisons between the deviations and the standard error 
of the two estimators, $\theta_{SSA}^{(N)}$ and $\theta_K^{(N)}$, together with Corollary 1, show that the improvement of the bias does not lead does not lead to lead to a substantial increase of error since the range of bias improvement and the change in dispersion are the same.
Our numerical experiments show that  no increase of variance occur. 
Further investigations in the future might bring sharper evaluation of the error those estimators have.    

Now, we consider the situation when convexity of $F$ with respect to the second argument might not be present. 
\begin{theorem}
\label{kernel-bias-smooth}
Assume that the kernel $K$ satisfies (k1)-(k2). Suppose that the function $F(u,\cdot)$ is continuously differentiable and its gradient $\nabla F(u,x)$ has a modulus of continuity of form 
\[
\|\nabla F(u,x+y) - \nabla F(u,x)\| \leq \sum_{j=1}^\ell L_j\|y\|^{\alpha_j} 
\]
with ${\alpha_j}\in(0,1]$, $\ell\geq 1$.
Then constant $\alpha\in (0,1]$ and $L>0$ exist such that for bandwidth $h_N<1$, we have 
\begin{align}
&\;\big|\theta_K^{(N)} -\theta\big|\leq \big|\theta_{SAA}^{(N)}-\theta \big|+ Lh_N^{1+\alpha}\; \text{ a.s. and, hence,}\\
&\;\big|\Eb[\theta_K^{(N)}] -\theta\big|\leq \big|\Eb[\theta_{SAA}^{(N)}]-\theta \big|+ Lh_N^{1+\alpha}\label{dispersion1}\\
 &\; \Big(\Eb[(\theta_{K}^{(N)} -\Eb[\theta_{K}^{(N)}] )^2]\Big)^{\frac{1}{2}} \leq \Big(\Eb (\theta_{SAA}^{(N)}-\Eb[\theta_{SAA}^{(N)}] )^2\Big)^{\frac{1}{2}} +  2Lh_N^{1+\alpha}\label{deviation1}\\
&\;\Big(\Eb\Big[\big(\theta_K^{(N)} - \theta  \big)^2\Big]\Big)^{\frac{1}{2}}\leq \Big(\Eb\Big[\big(\theta_{SAA}^{(N)} - \theta  \big)^2\Big]\Big)^{\frac{1}{2}} + Lh_N^{1+\alpha}.
\label{MSE-upper1}
\end{align}
\end{theorem}
\begin{proof}
Denote the optimal solution of the estimation problem in \eqref{p:kernel} by $\hat{u}_K^{(N)}$.
Using the mean-value theorem and substituting $z= \frac{x-X_i}{h_N}$, we obtain for any $i=1,\dots, N$ the following chain of relations: 
\begin{align}
\int & F(u,x)K\Big(\frac{x-X_i}{h_N}\Big)\frac{1}{h_N^m}dx  = \int F(u,X_i+h_Nz)K(z)dz\notag\\
 &= \int \Big( F(u,X_i) + h_N \big\langle\nabla_x F(u,X_i+ \zeta_i h_Nz),z\big\rangle \Big)  K(z)dz\notag\\
 & = F(u,X_i) + \int h_N \big\langle\nabla_x F(u,X_i),z\big\rangle  K(z)dz \\
 &\quad  +\int  h_N \big\langle\nabla_x F(u,X_i+ \zeta_i h_Nz)-\nabla_x F(u,X_i),z\big\rangle K(z)dz\notag\\
 & = F(u,X_i) + \int h_N \big\langle\nabla_x F(u,X_i+ \zeta_i h_Nz)-\nabla_x F(u,X_i),z\big\rangle K(z)dz. \label{gradient-intermed}
\end{align}
Here $\zeta_i\in (0,1)$, $i=1,\dots, N$. Using Cauchy-Schwarz inequality in \eqref{gradient-intermed} and the modulus of continuity for the gradient, we get
\[
\big|\big\langle\nabla_x F(u,X_i+ \zeta h_Nz)-\nabla_x F(u,X_i),z\big\rangle \big|\leq \sum_{j=1}^\ell L_j h_N^{\alpha_j}\|z\|^{1+\alpha_j}.
\]
Assuming $h_N<1$, we use $\alpha$ equal to the smallest of $\alpha_j$ to obtain the estimate
\[
\big|\big\langle\nabla_x F(u,X_i+ \zeta h_Nz)-\nabla_x F(u,X_i),z\big\rangle\big| \leq h_N^{\alpha}\sum_{j=1}^\ell L_j\|z\|^{1+\alpha_j}.
\]
Using this inequality and \eqref{gradient-intermed}, we obtain 
\begin{align*}
\theta_K^{(N)}-\theta_{SAA}^{(N)} &= 
\min_{u\in U} \Bigg[\frac{1}{N}\sum_{i=1}^N \int F(u,X_i+h_Nz)K(z)dz\Bigg] -
\min_{u\in U} \Big[\frac{1}{N}\sum_{i=1}^N F(u, X_i)\Big]\\
&\geq 
\frac{1}{N}\sum_{i=1}^N \int F(u_K^{(N)},X_i+h_Nz)K(z)dz -
\frac{1}{N}\sum_{i=1}^N F(u_K^{(N)}, X_i)\\
 &=  
\frac{1}{N}\sum_{i=1}^N \int \Big(F(u_K^{(N)},X_i+h_Nz)- F(u_K^{(N)}, X_i) \Big) K(z)dz\\
&=  
\frac{1}{N}\sum_{i=1}^N \int h_N \big\langle\nabla_x F(u,X_i+ \zeta_i h_Nz)-\nabla_x F(u,X_i),z\big\rangle K(z)dz \\
&\geq -Lh_N^{1+\alpha},
\end{align*}
where $L= \sum_{j=1}^\ell L_j \bar{m}_{1+\alpha_j}(K)$.
On the other hand, we can use an optimal solution of the SAA problem denoted $\hat{u}_{SAA}^{(N)}$ in a similar way.  Then we obtain an upper bound for the difference of the kernel and empirical estimate:
\begin{align*}
\theta_K^{(N)}-\theta_{SAA}^{(N)}
&\leq 
\frac{1}{N}\sum_{i=1}^N \int F(u_{SAA}^{(N)},X_i+h_Nz)K(z)dz -
\frac{1}{N}\sum_{i=1}^N F(u_{SAA}^{(N)}, X_i)\\ 
&\leq Lh_N^{1+\alpha}.
\end{align*}
Putting the two bounds together, we obtain
\begin{equation}
\label{mainbounds}
-Lh_N^{1+\alpha}\leq \theta_K^{(N)} - \theta_{SAA}^{(N)} \leq Lh_N^{1+\alpha}\; \text{ a.s.}
\end{equation}
which implies the first claim in the theorem.
Relation \eqref{dispersion1} follows by taking expectation on both sides of inequality \eqref{mainbounds}.

The comparison of the error \eqref{MSE-upper1} follows from \eqref{mainbounds} by using the argument about the norm of Banach lattices for positive random variables in the same way as for \eqref{MSE-upper}. 
Finally, from \eqref{mainbounds}, we obtain
\[
-Lh_N^{1+\alpha}\leq -\Eb[\theta_K^{(N)}] +\Eb[\theta_{SAA}^{(N)}] \leq Lh_N^{1+\alpha}\; \text{ a.s.}
\]
Adding this to \eqref{mainbounds} entails
\[
-2Lh_N^{1+\alpha}\leq \theta_K^{(N)}-\Eb[\theta_K^{(N)}] - \big(\theta_{SAA}^{(N)}-\Eb[\theta_{SAA}^{(N)}]\big) \leq 2Lh_N^{1+\alpha}\; \text{ a.s.}
\]
Using the triangle inequality for the norm in $\Lc_2$ as in the proof of \eqref{deviation}, we infer from the last displayed inequality the relation \eqref{deviation1}. 
\end{proof}

Assuming further smoothness for the objective function and taking a special kernel, we could provide an estimate with a smaller bias than the one given by SAA.

\begin{theorem}
\label{kernel-bias-smooth2}
Consider the uniform kernel $K$ on the unit ball $B$ in $\Rb^m$. Assume that the function $F(u,\cdot)$ is twice continuously differentiable with the following properties
\begin{itemize}
\item $\Eb[\nabla^2_x F(u,X)]$ has a trace bounded from below by $\bar{\lambda} >0$ for all $u\in U$;
 \item $\nabla^2_x F(u,x)$ has a modulus of continuity of form 
\[
\|\nabla^2_v F(u,x+y) - \nabla^2_x F(u,x)\| \leq \sum_{j=1}^\ell L_j\|y\|^{\alpha_j}
\]
with $\alpha_j\geq 0$, $j=1,\dots\ell$, $\ell\geq 1$.
\end{itemize}
Then a positive number $h^*$ exists, such that whenever $h_N\in (0,h^*)$ the bias of the kernel-based estimator is smaller than the one of SAA, i.e., relation \eqref{bias-interval-final} holds.
\end{theorem}
\begin{proof}
Denote the optimal solution of the estimation problem in \eqref{p:kernel} by $\hat{u}_K$. The kernel is defined by 
\[
 K(z) = \begin{cases} \frac{1}{\text{vol}(B)} & \ \text{if}\ \|z\| \le 1,\\
 0 & \text{otherwise}
 \end{cases} 
 \quad\text{where }\; \text{vol}(B)=\frac{\pi^{m/2} } {\Gamma(1+m/2)}
 \]
with $\text{vol}(B)$ standing for the volume of the $m$-dimensional unit ball determined by the Euclidean norm. 
Using the mean-value theorem, we obtain with some $\zeta_i\in (0,1)$ for any $i=1,\dots, N$ the following chain of equations:
\begin{align*}
\int_B & F(u,X_i+h_Nz)K(z)dz\\
 &= \int_B \Big( F(u,X_i) + h_N \big\langle\nabla_x F(u,X_i),z\big\rangle + \frac{h_N^2}{2}\big\langle z, \nabla^2_xF(u,X_i+\zeta_i h_Nz) z \big\rangle\Big)\; K(z)dz\\
 &= F(u,X_i) + \frac{h_N^2}{2\text{vol}(B)}\int_B \big\langle z, \nabla^2_xF(u,X_i+\zeta_i h_Nz) z \big\rangle\; dz \\
 &= F(u,X_i) + \frac{h_N^2}{2\text{vol}(B)}\int_B \big\langle z, \nabla^2_xF(u,X_i) z \big\rangle\; dz \\
 &\qquad + \frac{h_N^2}{2\text{vol}(B)}\int_B \big\langle z, \big[\nabla^2_xF(u,X_i+\zeta_i h_Nz) - \nabla^2_xF(u,X_i)\big]z \big\rangle\; dz. 
\end{align*}
In this chain, the second equation holds due to (k1).
The last term at the right-hand side can be estimated by the modulus of the Hessian in a similar ways as in the previous theorem assuming that $h_N<1.$ We have 
\[
- L \zeta^\alpha h_N^\alpha\leq \int_B \big\langle z, \big[\nabla_x^2F(u,X_i+\zeta h_Nz) - \nabla^2_xF(u,X_i)\big]z \big\rangle\; dz \leq L\zeta^\alpha h_N^\alpha.
\]
Using an optimal solution $\hat{u}_K$ of problem \eqref{p:kernel}, we obtain the following estimate
\begin{align*}
\Eb[\theta_K] & - \Eb[\theta_{SAA}^{(N)}]  \\
&= 
\mathbb{E} \min_{u\in U} \Big[\frac{1}{N}\sum_{i=1}^N \int_B F(u,x)K\Big(\frac{x-X_i}{h_N}\Big)\frac{1}{h_N^m}dx\Big] -
\mathbb{E} \min_{u\in U} \Big[\frac{1}{N}\sum_{i=1}^N F(u, X_i)\Big]\\
& \geq \Eb \left[\frac{1}{N}\sum_{i=1}^N \Big( \int_B F(\hat{u}_K,X_i+h_Nz)K(z) dz -  F(\hat{u}_K, X_i)\Big)\right]\\
& \geq 
\frac{h_N^2}{2\text{vol}(B)} 
 \int_B \langle z, \Eb[\nabla^2_x F(\hat{u}_K,X)] z \rangle dz - \frac{Lh_N^{2+\alpha}}{2\text{vol}(B)}.
\end{align*}
The sign of the fist term depends on the sign of the trace of the matrix
$H=\Eb[\nabla^2_x F(\hat{u}_K,X)]$, which is assumed positive. To see this, recall that 
$\nabla^2_x F(\hat{u}_K,X)$ is symmetric and thus $H$ is diagonalizable and representable as 
\[
H = \sum_{i=1}^m \lambda_i y^i (y^i)^\top,
\]
where $\lambda_i$ are the eigenvalues and the $y^i$ are the respective eigenvectors with $y^i$, $i=1,\dots m$ constituting an orthonormal basis of $\Rb^m.$
We have $\langle z, H z \rangle = \sum_{i=1}^m \lambda_i (\langle z, y^i \rangle)^2$ and, thus, 
\[
\int_B \langle z, \Eb[\nabla^2_x F(\hat{u}_K,X)] z \rangle dz =
\sum_{i=1}^m \lambda_i \int_B (\langle z, y^i \rangle)^2\, dz
\]
The integral $r=\int_B (\langle z, y^i \rangle)^2\, dz$ can be calculated exactly but it is sufficient to observe that its value $r$ is the same for all $i$ due to the isotropy of the unit ball (w.r.to the Euclidean norm). 
Therefore, 
\[
\int_B \langle z, \Eb[\nabla^2_x F(\hat{u}_K,X)] z \rangle dz = r\sum_{i=1}^m \lambda_i\,\geq r\bar{\lambda}>0.
\]
We obtain
\[
\Eb[\theta_K] - \Eb[\theta_{SAA}^{(N)}]  \geq 
\frac{h_N^2}{2\text{vol}(B)}\Big( r\bar{\lambda}- Lh_N^\alpha\Big)
\]
Therefore, the right hand side of the inequalities will be positive for sufficiently small bandwidth, yielding
\begin{equation}
\label{bias-uniformK-lowerbound}
\Eb[\theta_K] \geq \Eb[\theta_{SAA}^{(N)}] \text{ a.s.}
\end{equation}

In a similar way, using the optimal solution $\hat{u}$ of problem \eqref{p:basic}, we can obtain an upper bound
for the bias of the kernel smoothing problem
\begin{align}
\Eb[\theta_K] - \theta  
& =
\Eb\Big[\min_{u\in U}\frac{1}{N} \sum_{i=1}^N \int_B  F(u,X_i+h_Nz)K(z)dz \Big]-
\min_{u\in U} \frac{1}{N} \sum_{i=1}^N \Eb[F(u,X_i)]\notag\\
& \leq \Eb \left[\frac{1}{N}\sum_{i=1}^N \Big( \int_B F(\hat{u},X_i+h_Nz)K(z) dz -  F(\hat{u}, X_i)\Big)\right]\notag\\
& \leq 
\frac{h_N^2}{2\text{vol}(B)} 
 \int_B \langle z, \Eb[\nabla^2_x F(\hat{u},X)] z \rangle dz + \frac{Lh_N^{2+\alpha}}{2\text{vol}(B)}\label{bias-uniformK-upperbound}.
\end{align}

The expression at the right hand side is smaller that
$\theta - \Eb[\theta_{SAA}^{(N)}]>0$ for sufficiently small bandwidth $h_N$.
Therefore, putting \eqref{bias-uniformK-lowerbound} and  \eqref{bias-uniformK-upperbound}
\[
\Eb[\theta_{SAA}^{(N)}]  - \theta\leq \Eb[\theta_K] - \theta \leq \theta - \Eb[\theta_{SAA}^{(N)}],
\]
which concludes the proof.
\end{proof}

One question discussed extensively in the statistic literature is about the choice of the parameter $h_N$.
The discussion in \cite[Section 1.2.4]{tsybakov2008introduction} shows that the approach on determining optimal bandwidth for a fixed data density can be seriously criticized. Our numerical experiments show that in the context of optimization, even if the data is normally distributed, that is, $X$ has a density and we select the recommended bandwidth for the kernel density estimator, the result may not be the best. 
One could use the rule of thumb for the normal distribution as a good practical choice. The bandwidth $h$ can be selected according to the Sheather-Jones plug-in method or the Silverman’s rule of thumb \cite{sheather2004density}. The  Sheather-Jones plug-in method recommends to set the bandwidth $h=1.06 \hat{\sigma} N^{-1/5}$, where $\hat{\sigma}$ is the sample standard deviation. The Silverman's rule recommends the bandwidth be chosen as $h=0.9 a N^{-1/5}$, where $a= \min\{\hat{\sigma}, \text{(sample interquartile range)}/1.34\}.$

Another practical procedure might be the following. Corollary 1 shows that the bias of the kernel estimator becomes smaller than the bias of the SAA if $Lh_N^\alpha$ is smaller than the absolute value of the bias of the SAA estimator. The constants $L$ and $\alpha$ can be evaluated independently of the random data using the modulus of continuity and the moments of the chosen kernel. We could estimate the bias of the SAA problem  if we have a way of evaluating $\Eb[F(\bar{u}, X)]$ more precisely at a given fixed point $\bar{u}.$
We know that
\[
\theta_{SAA}^{(N)} - \theta \leq \frac{1}{N} \sum_{i=1}^N F(\hat{u},X_i)- \Eb[F(\hat{u}, X)], 
\] 
where $\hat{u}$ is an optimal solution of problem \eqref{p:basic}.
Therefore, we could solve the SAA problem (potentially, using a smaller sample), then set
$\bar{u}$ to be the obtained solution and estimate $\Eb[F(\bar{u}, X)]. $ 
Clearly, $\Eb[F(\bar{u}, X)]\geq \Eb[F(\hat{u}, X)].$ Therefore, we obtain statistical bounds for
$\theta$ as follows: 
$\frac{1}{N} \sum_{i=1}^N F(\bar{u},X_i)\leq \theta\leq \Eb[F(\bar{u}, X)].$
We can use this relation to derive an upper bound for 
$Lh_N^\alpha$ and subsequently for the bandwidth. 

Another approach might be based on the asymptotic analysis of the SAA solution.
It is known that the bias of $\theta_{SAA}^{(N)} $  is of order $o(N^{-\frac{1}{2}})$ if 
the true problem has a unique solution while the order is $O(N^{-\frac{1}{2}})$ if 
the true problem has multiple solutions (cf. \cite[p 186]{mainbook}). Therefore, we may choose the bandwidth according to $O(N^{-\frac{1}{2}-\epsilon})$ for $\epsilon>0.$
The question of what the best value for $h_N$ is remains open.

\section{Applications}
\label{s:applications}

Most notable examples, where our results hold are given by the linear regression models, LASSO, ridge regression models, binary classification, or risk estimation problems. We discuss some of them in this section.

\subsection{LASSO problem}
Consider a sample of size $N$ with an outcome $Y$ and explanatory variables comprised in an $m$-dimensional vector $X$. Then the objective of LASSO is to solve
\[
\min _{\beta_0,\beta\in\Rb^m } \sum _{i=1}^{N}\big(y_{i}-\beta _{0}-\beta^\top X^{i} \big)^{2}\; \text{ subject to }\;\sum _{j=1}^{m}|\beta_{j}|\leq t.
\]
One can view this formulation as the SAA for minimizing the expected 
squared error of the model $y=\beta_{0}+\beta^\top x$. 

Without loss of generality, we may consider $\beta_0=0$ and denote $\tilde{\beta}= (-1,\beta)$ and $\tilde{x}=(y,x)$ with $\tilde{X}^i=(Y^i,X^i,1)$, $i=1,\dots,N.$
The smoothed (with respect to the data) problem is the following:
\begin{equation}
\label{regression-smoothed}
\min _{\tilde{\beta}} \sum _{i=1}^{N}\int\big(\tilde{\beta}^\top \tilde{x} \big)^{2} K\Big(\frac{\tilde{X}^i-\tilde{x}}{h_N}\Big)\frac{1}{h_N^{m+1}}\,d\tilde{x}, \quad \text{ subject to }\;\sum _{j=1}^{m}|\beta_{j}|\leq t.
\end{equation}

We shall verify the assumptions of Theorem~\ref{t:kernel-bias-convex} and Theorem~\ref {t:K-consistency2}  for a given sample.  
Here 
\[
F(\tilde{\beta},\tilde{x}) = \big(\tilde{\beta}^\top \tilde{x} \big)^{2} 
.\]
In order to define the modulus of continuity of $F(\tilde{\beta},\tilde{x})$, we see that
\begin{align*}
|F(\tilde{\beta},\tilde{x} +z) &- F(\tilde{\beta},\tilde{x})| = 
|\tilde{\beta}^\top z \tilde{\beta}^\top (2\tilde{x} +z)|\\
& \leq \|\tilde{\beta}\|\| z \|\big(2\|\tilde{\beta}\| \|\tilde{x}\| +\|\tilde{\beta}\|\|z\|)\big) \leq c_1\|z\|(2c_1c_2 + c_1\|z\|).
\end{align*}
In the last expression the constants $c_1$ and $ c_2$ refer to the maximum of the norms of the vectors $\tilde{\beta}$ and $\|\tilde{x}\|$, respectively. Those maxima are finite due to the constraint in LASSO and the finite number of observation in the sample at hand. 
Therefore, we can define the modulus of continuity to be 
\[
w(t) = c_1^2t(2c_2 + t)\,\text{ for } t\geq 0.
\]
The function $w(\cdot)$ is non-decreasing for $t\geq 0$ and $\lim_{t\downarrow 0} w(t) =0.$
If we use a kernel $K$ satisfying (k1)-(k2), then $w$ is $K$-integrable. 
This verifies assumption (c1k). 
In this example, assumption (c2) requires the existence of an integrable envelope function for the
set $f_{\tilde{\beta}}(x) = \big(\tilde{\beta}^\top \tilde{x} \big)^{2}$. This assumption is satisfied due to the compactness of the set of all feasible vectors $\tilde{\beta}$ when the vector $X$ has a finite second moment. Assumptions (c3) and (c4) are evident and, hence, Theorem~\ref {t:K-consistency2} applies.

Furthermore, we easily see that $F(\tilde{\beta},\tilde{x}) = \big(\tilde{\beta}^\top \tilde{x} \big)^{2} $ is convex with respect to the second argument. This shows that Theorem~\ref{t:kernel-bias-convex} is applicable as well.

To solve the smoothed problem efficiently, we should be able to calculate the integrals easily and precisely, which can be done in special cases. For example,  we can easily implement smoothing by the normal kernel $K(z) = \frac{e^{-\frac{1}{2}z^\top A^{-1} z}}{\sqrt{(2\pi)^m|A|}} $, where $A$ is a positive definite covariance matrix with the identity matrix being a particular case. 

We observe that 
\begin{multline}
\label{regression-reg}
\int\big(\tilde{\beta}^\top \tilde{x} \big)^{2} K\Big(\frac{\tilde{X}^i-\tilde{x}}{h_N}\Big)\frac{1}{h_N^{m+1}}\,d\tilde{x} = \int\big(\tilde{\beta}^\top (\tilde{X}^i +h_Nz)\big)^{2} K(z)\,dz\\
 = \big(\tilde{\beta}^\top (\tilde{X}^i)\big)^{2} 
+ 2 h_N\tilde{\beta}^\top \tilde{X}^i \int \tilde{\beta}^\top z\,K(z)\,dz +
h_N^2 \int \big(\tilde{\beta}^\top z\big)^2\,K(z)\,dz\\
= \big(\tilde{\beta}^\top (\tilde{X}^i)\big)^{2}  +
h_N^2\int \big(\tilde{\beta}^\top z\big)^2\,K(z)\,dz
\end{multline}
The last equality follows by (k1). 
Thus, the integral represents the expected value of the random variable $W=(\tilde{\beta}^\top Z)^2$ when $Z$ is a normally distributed random vector with expected value 0 and covariance matrix $A$. Hence $W/\sigma^2$ with $\sigma^2 = \tilde{\beta}^\top A\tilde{\beta}$ is a $\chi^2$-random variable with one degree of freedom. 
Therefore, we continue \eqref{regression-reg} as follows:
\[
\int\big(\tilde{\beta}^\top \tilde{x} \big)^{2} K\Big(\frac{\tilde{X}^i-\tilde{x}}{h_N}\Big)\frac{1}{h_N^{m+1}}\,d\tilde{x}  = \big(\tilde{\beta}^\top (\tilde{X}^i)\big)^{2}  + h_N^2 \tilde{\beta}^\top A \tilde{\beta}
\]
We can summarize the observation in the following corollary. 
\begin{corollary}
The regression problem \eqref{regression-smoothed} with the normal kernel with covariance $A$ is equivalent to a regularization in the objective, i.e., 
\[
\min _{\beta\in\Rb^m } \sum _{i=1}^{N}\big(y_{i}-\beta^\top X^{i} \big)^{2}
+ \kappa \|(-1,\beta)\|_A^2\; \text{ subject to }\;\sum _{j=1}^{m}|\beta_{j}|\leq t.
\]	
\end{corollary}
Here, we have set $\kappa=h_N^2$ and the norm is the weighted Euclidean norm.

\subsection{Classification problems}

We also examine an SVM problem, using the $L_1$ norm  of the classification error. 
For simplicity, we only discuss binary classification. 
We have labeled data from two classes: $m_1$ data points from $S_1$, denoted $x^j$, $j=1,\dots m_1$ and $m_2$ points from $S_2$,  denoted $y^j$, $j=1,\dots m_1$.
We have measured $n$ features and seek to determine a  linear classifier 
$\varphi:\R^n\to\R$ such that $\varphi(x) < 0$ for all $x\in S_1$,
and $\varphi(x)>0$ for all $y\in S_2$.  
The classifier is defined by setting 
$\varphi(z)= v^\top z -\gamma \text{ for any } z\in\mathbb R^n$. 
The following optimization problem determines the classifier
\begin{equation}
\label{eq:svmE1}
\begin{aligned}
\min_{v,\gamma}\, &\,  \frac{1}{m_1} \sum_{j=1}^{m_1} \max\{0, \langle v, x^j\rangle - \gamma \} + \frac{1}{m_2} \sum_{j=1}^{m_2} \max\{0, \gamma - \langle v, y^{j}\rangle\}  \\
\text{s. t. }\, & \, \|v\|=1.
\end{aligned}
\end{equation}
This problem is a sample average approximation for the problem of minimizing the expected misclassification.  
In a more general risk-sensitive setting, the problem is investigated in \cite{Dentcheva2019}.
The kernel-based problem has the following form
\begin{equation}
\label{eq:svm_kernel}
\begin{aligned}
\min_{v,\gamma}\, &\,  \frac{1}{m_1} \sum_{j=1}^{m_1} \int_{\Rb^n} \max\{0, \langle v, x^j + h_{m_1} z \rangle - \gamma \} K(z)\;dz\\
 &\qquad\qquad+ \frac{1}{m_2} \sum_{j=1}^{m_2} \int_{\Rb^n} \max\{0,\gamma - \langle v, y^{j}+ h_{m_2} z\rangle\} K(z)\;dz \\
\text{s. t. }\, & \, \|v\|=1.
\end{aligned}
\end{equation}
We shall verify that the problem \eqref{eq:svm_kernel} satisfies the conditions of 
Theorem~\ref {t:K-consistency2}  and of Theorem~\ref{t:kernel-bias-convex}.
We observe that this problem contains estimation of the sum of two expected values and does not fit perfectly in the general problem formulation of \eqref{p:kernel}. However, the necessary modification is straightforward.
Here 
\[
F_1(v,\gamma,x) = \max\{0, \langle v, x \rangle + \gamma\}\quad
F_2(v,\gamma,x) = \max\{0, \gamma-\langle v, x \rangle\}
\]
and we minimize $\Eb\big[F_1(v,\gamma,x)\big] + \Eb\big[F_2(v,\gamma,x)\big].$
The decision $v$ is in a compact space due to the restriction on its norm. Without loss of generality, we may assume that $\gamma$ is bounded in a compact interval $[-c,c]$ for sufficiently large constant $c>0$.
Both $F_1$ and $F_2$ are continuous with respect to the first argument and convex with respect to the second argument.
Integrable envelope functions for the
sets $f^1_{v,\gamma}(x) = \max\{0, \langle v, x \rangle+ \gamma\}$ and $f^2_{v}(x) = \max\{0, \gamma-\langle v, x \rangle\}$ exist due to the compactness of the feasible set and the integrability of $X$. 
It remains to consider the modulus of continuity of $F_i(v,x)$, $i=1,2$. We see that
\begin{align*}
|F_i(v,\gamma,x +z) - F_i(v,\gamma,x)| \leq |v^\top z|
\leq \|v\|\| z \|  \leq \|z\|,\quad i=1,2.
\end{align*}
In the last inequality, we have used the constraint on $v$.  
Therefore, the modulus of continuity  is $w(t)= t$ for all $t\geq 0.$
The necessary conditions on $w(\cdot)$ which are required in (c1k), are easy to check when the kernel $K$ satisfies (k1)-(k2).  
We conclude that Theorem~\ref {t:K-consistency2} and Theorem~\ref{t:kernel-bias-convex} are both applicable.

\subsection{Portfolio optimization with Average Value-at-Risk}
\label{avar-port}

Another example, in which our results are applicable is the following problem.
Let $X$ be a random vector, representing the random returns of $m$ securities and $u$ be a decision regarding allocation of the available capital $K$. The set $U$ comprises various restrictions
on our allocations, e.g.,
\[
U=\{ u\in \Rb^m :\; \sum_{i=1}^m u_i =K,\; l_i\leq u_i \leq  b_i\},
\]
where $l_i$ and $u_i$ are lower and upper bounds, respectively, for the investment in the $i$-th security.
 We may optimize a combination of the mean return with the Average Value-at-Risk at some level $\beta\in (0,1)$ to determine our portfolio. 
Our objective is to minimize the expected value of the function
\[
F(u,\eta,x) = -\kappa\langle u, x\rangle + (1-\kappa) \Big(-\eta + \frac{1}{\beta} \max\big\{0, \eta - \langle u, x\rangle\big\}\Big),
\] 
where $\kappa \in (0,1)$ and $\beta\in (0,1)$ is the chosen probability level for the Average Value at Risk. The optimization problem reads:
\[
\min_{u\in U, \eta\in\Rb} \Eb\big[F(u,\eta,X)\big].
\]
We can restrict $\eta$ to a properly chosen interval without affecting the optimal (unrestricted) solution since $\eta$ is the $(1-\beta)$- quantile of the distribution of $\langle u, X\rangle$ and $\beta\in (0,1).$  Therefore, we can assume that the compactness requirement in (c4) is satisfied. 

In order to define the modulus of continuity of $F(u,\eta,x)$, we see that
\begin{align*}
|F(u,\eta,x+z) & - F(u,\eta,x) |  = \\
&\Big|-\kappa\langle u, z\rangle + \frac{1-\kappa}{\beta}\Big( \max\big\{0, \eta - \langle u, x+z\rangle\big\} - \max\big\{0, \eta - \langle u, x\rangle\big\}\Big)\Big|\\
& \leq \kappa\|u\|\| z \| + \frac{1-\kappa}{\beta}|\langle u, z\rangle| \\
& \leq c\|z\|(\kappa+ \frac{1-\kappa}{\beta}).
\end{align*}
In the last expression, the constants $c$ is the maximum of the norm of $u\in U$. 
Therefore, we can define the modulus of continuity to be $w(t) = c(\kappa+ (1-\kappa)/{\beta})t$ for $t\geq 0.$
The necessary conditions on $w(\cdot)$ which are required in (c1k), are easy to see when the kernel $K$ satisfies (k1)-(k2). 
The existence of an integrable envelope function for the
set $f_{u,\eta}(x) = -\kappa\langle u, x\rangle + (1-\kappa) \big(-\eta + \frac{1}{\beta} \max\big\{0, \eta - \langle u, x\rangle\big\}\big)$ is ensured by the compactness of the feasible set and the integrability of $X$.  Hence, Theorem~\ref {t:K-consistency2} applies.
Furthermore, we easily see that $F(u,\eta, \cdot)$ is convex, which shows that Theorem~\ref{t:kernel-bias-convex} is applicable as well.

\section{Numerical Results}
\label{s:numerical_results}

We have selected to estimate the Average Value at Risk (AVaR$_\alpha$) from various distributions and drawing samples of different sizes. For a fix parameter $\alpha\in (0,1)$, in-line with our theoretical discussion, we denote the true value of AVaR$_{\alpha}$ by $\theta.$
We use the variational representation of Average Value at Risk to calculate it
\begin{equation}
    \theta= \min_{z \in \mathbb{R}}\left\{  z+ \frac{1}{\alpha}\Eb[{(X-z)}_{+}]\right\}.
\end{equation}
The kernel-based problem has the form:
\begin{equation}
    \theta_{K}^{(N)}= \min_{z \in \mathbb{R}} \left\{z +\frac{1}{\alpha} \sum_{i=1}^n \int \max \{0, x-z\}K(\frac{x-X_i}{h_m})dx \right\}
\end{equation}

In the first sequence of experiments, we have generated observation $X_i$, $i=1,\cdots,N$ from  a normal distribution $\mathcal{N}(10,3)$ and have selected values of $\alpha=0.05$ and $\alpha=0.2$. The true value of $\theta$ for the normal distribution $\mathcal{N}(10,3)$ is $10 + \frac{\sqrt{3} e^{-1.36}}{0.05\sqrt{2\pi}}$ when $\alpha=0.05$  and $10 + \frac{\sqrt{3} e^{-0.36}}{0.2\sqrt{2\pi}}$, and $\alpha=0.2.$ 
First, we use the uniform kernel:  
\[
K(z)=\frac{1}{2h_N}\quad\text{with support } |z| \leq h_N.
\]
We solve the SAA problem and the kernel-based problem for $1000$ replications of samples with size $N$. Using the calculated optimal values, we estimate the bias and the variance of $\theta^{(N)}_{K}$, $\theta^{(N)}_{SAA}$.
In these experiments, we solve the kernel-based problems by setting the bandwidth according to the Sheather-Jones method (denoted S--J Rule), the Silverman's rule (denoted S Rule), and for the values 0.5, 0.35, 0.2, and 0.05.
The best choice of bandwidth in terms of bias is shown in boldface. Notice that for those choice, the variance of may not be the best but it is always smaller than the variance of SAA. 
The results are shown in the following tables.\\

\begin{center}
Bias\\ \vspace*{1ex}
{\small
\!\!\!\begin{tabular}{|c| c |c | c| c | c |c| c|}
 \hline\diagbox{N}{$h_N$} & S Rule & S--J Rule & 0.5  & 0.35 & 0.2 & 0.05 & SAA\\ 
 \hline
 100 &  \textbf{-0.0003} &0.0301  &-0.0278& -0.0563& -0.0764& -0.0870 & -0.0878 \\
 200 &  0.0218 &  0.0446&\textbf{0.0134}& -0.0133& -0.0316& -0.0411 & -0.0419 \\
 500 & 0.0248 &0.0403& 0.0345& 0.0090& \textbf{-0.0078}& -0.0160 & -0.0167\\
  \hline
\end{tabular}
}
\end{center}
\bigskip

\begin{center}
Variance\\ \vspace*{1ex}
{\small
\begin{tabular}{|c| c |c | c| c | c |c|c|}
 \hline\diagbox{N}{$h_N$}& S Rule & S--J Rule & 0.5  & 0.35 & 0.2 & 0.05 & SAA\\ 
 \hline
100 & 0.1726& 0.1698 &0.1717& 0.1768&  0.1806& 0.1828 & 0.1830 \\
200 &  0.0882& 0.0872& 0.0867 & 0.0891&  0.0908& 0.0916 & 0.0917 \\
500 &  0.0396 & 0.0393&0.0385& 0.0396&  0.0402& 0.0406 &  0.0406\\
  \hline
\end{tabular}
}
\end{center}
\bigskip

We can see that  Sheather-Jones method for selecting the bandwidth does not always work well in the context of minimization. Perhaps the bandwidth is too large when the sample size increases. The Silverman’s rule works much better in terms of bias reduction. It also has smaller variance than the SAA.  
When we reduce the bandwidth from $0.5$ to $0.05$, the bias shifts from a small positive number to a negative one approaching the bias of SAA. At the same time, the variances of the estimators do not increase. In this case, when we choose bandwidth from $0.05$ to $0.35$, the kernel estimator can reduce the bias. 


In the second experiment, we choose the Epanechnikov kernel function 
\[
K(z)=\frac{3}{4}(1-z^2),\quad\text{with support }|z| \leq h_N.
\]
We solve the SAA problem and the kernel-based problem for $1000$ replications of samples with size $N$. Using the results, we estimate the bias and the variance of $\theta^{(N)}_{K}$, $\theta^{(N)}_{SAA}$. We solve the kernel-based problems by choosing the bandwidth in the same way as in the previous experiments.
The results are shown in the following tables.\\

\begin{center}
Bias\\ \vspace*{1ex}
{\small
\begin{tabular}{|c | c  |c | c| c | c |c| c|}
 \hline \diagbox{N}{$h_N$}& S Rule & S--J Rule & 0.5  & 0.35 & 0.2 & 0.05 & SAA\\ 
 \hline
 100 &   -0.0326 &\textbf{-0.0135}& -0.0496& -0.0679& -0.0807& -0.0873 &-0.0878\\ 
 200 &  \textbf{-0.0023 }&0.0118& -0.0074& -0.024& -0.0354& -0.0414  &-0.0419\\
 500 & 0.0087 &0.0181& 0.0146& \textbf{-0.0009}& -0.0112& -0.0163 &-0.0167\\
  \hline
\end{tabular}
}
\end{center}
\bigskip
\begin{center}
Variance\\ \vspace*{1ex}
{\small
\begin{tabular}{|c |c |c | c| c | c |c| c|}
 \hline \diagbox{N}{$h_N$}& S Rule & S--J Rule & 0.5  & 0.35 & 0.2 & 0.05 & SAA\\ 
 \hline
 100 &   0.1760& 0.1740& 0.1761& 0.1790& 0.1814& 0.1829 &0.1830\\ 
 200 &  0.0895& 0.0888& 0.0886& 0.0901& 0.0912& 0.0916 &0.0917\\ 
 500 & 0.0400 &0.0398& 0.0393& 0.04& 0.0404& 0.0406 &0.0406\\ 
  \hline
\end{tabular}
}
\end{center}

\bigskip

In the third experiment, we consider the $\alpha=0.2$. We use the uniform and Epanechnikov kernel functions and the choice of bandwidth as in the  previous experiments. 
The results of uniform kernel are shown in the following tables.\\

\begin{center}
Bias\\ \vspace*{1ex}
{\small
\begin{tabular}{|c |c  |c | c| c | c |c| c|}
 \hline \diagbox{N}{$h_N$} & S Rule & S--J Rule  & 0.5  & 0.35 & 0.2 & 0.05 & SAA\\ 
 \hline
 100 &  0.0123& 0.0324& \textbf{-0.0057}& -0.0237& -0.0359& -0.040 &-0.0425\\ 
 200 & 0.0298& 0.0448& 0.0243& 0.0071& \textbf{-0.0043}& -0.0098  &-0.0103\\ 
 500 & 0.0224& 0.0328& 0.0289& 0.0118& \textbf{0.0007}& -0.0045 &-0.0049\\ 
  \hline
\end{tabular}
}
\end{center}
\bigskip

\begin{center}
Variance\\ \vspace*{1ex}
{\small
\begin{tabular}{|c | c |c | c| c | c |c| c|}
 \hline \diagbox{N}{$h_N$}& S Rule& S--J Rule  & 0.5  & 0.35 & 0.2 & 0.05 & SAA\\ 
 \hline
 100 &   0.07& 0.07& 0.0683& 0.0692& 0.0698& 0.0701 &0.0701\\ 
 200 & 0.0341& 0.0341&0.0332& 0.0336& 0.0338& 0.0339  &0.0339\\ 
 500 & 0.0153& 0.0153& 0.0149& 0.0151& 0.0152& 0.0153 &0.0153\\ 
  \hline
\end{tabular}
}
\end{center}
\bigskip

This experiment shows that the Sheather-Jones method and Silverman’s rule may not perform well in the context of optimization. The SAA method performs better than these two rules when the sample size is $200$ or $500$. However, we observe that the kernel method has less bias than the SAA method when the bandwidth decreases from $0.5$ to $0.05.$
The results for the experiments with the Epanechnikov kernel are reported below.

\begin{center}
Bias\\ \vspace*{1ex}
{\small
\begin{tabular}{|c | c |c | c| c | c |c| c|}
 \hline \diagbox{N}{h} & S Rule& S--J Rule  & 0.5  & 0.35 & 0.2 & 0.05 & SAA\\ 
 \hline
 100 &  -0.0088& \textbf{0.0035}& -0.0198& -0.0303& -0.0384& -0.0422 &-0.0425\\ 
 200 &  0.0142& 0.0233& 0.0111& \textbf{0.0004}& -0.0066& -0.01  &-0.0103\\ 
 500 &  0.0119& 0.0179& 0.0155& 0.0052& \textbf{-0.0015}& -0.0046 &-0.0049\\ 
  \hline
\end{tabular}
}
\end{center}
\bigskip
\begin{center}
Variance\\ \vspace*{1ex}
{\small
\begin{tabular}{|c | c  |c | c| c | c |c| c|}
 \hline \diagbox{N}{h} &  S Rule & S--J Rule  & 0.5  & 0.35 & 0.2 & 0.05 & SAA\\ 
 \hline
 100 &  0.0700& 0.0700& 0.069& 0.0693& 0.0699& 0.0701& 0.0701\\ 
 200 &  0.0340& 0.0341& 0.0336& 0.0337& 0.0339& 0.0339&0.0339\\ 
 500 & 0.0153& 0.0153& 0.0151& 0.0152& 0.0152& 0.0153&0.0153\\ 
  \hline
\end{tabular}
}
\end{center}
\bigskip

From this experiment,  we can see that different kernels influence the performance of the bias reduction. When we choose the Epanechnikov kernel, as the sample size becomes larger, the reduction becomes more robust.
For example, when $N=500$ and the Epanechnikov kernel is employed, 
the bias ranges from $0.0155$ to $-0.0046$ when the bandwidth decreases from $0.5$ to $0.05$. In the same situations, the use of the uniform kernel results in a bias range from $0.0289$ to $-0.0045$.

The experiments show that the bandwidth is still a free parameter of the kernel estimation. At this time, we cannot give a precise recommendation for its best value based on our model. However, our theoretical results suggest to choose  reasonable small bandwidth to reduce the bias. Furthermore, we observe that smoothing does not increase the variance. 

\section{Conclusions}

We have discussed using smoothing with respect to the data in stochastic optimization problems of basic form, instead of the empirical measure. We have analyzed the consistency of the optimal value and its bias as an estimator. We have shown that smoothing reduces bias for many problems where convexity of the objective with respect to the data is present. The price to pay is a more complex estimator requiring to calculate integrals quickly and efficiently. As demonstrated in the discussion of regression problems, the choice of a smoothing measure may alleviate this difficulty depending on the application.  Further analysis and numerical experience will explore the limitations of this proposal. We believe that the effect of bias reduction is most essential for stochastic optimization problems related to cost and risk minimization in particular for practical situations where high-dimensional problems arise with  insufficient data. In such situations, precise evaluation of the optimal value becomes  more important.  

\bigskip

\noindent
{\small {\bf Acknowledgement:} The authors thank the associate editor and the two anonymous referee whose remarks helped improve the paper.

\end{document}